\documentclass[12pt]{article}
\usepackage{graphicx} 
\usepackage{color}
\usepackage{enumerate}
\usepackage{latexsym}
\usepackage[T1]{fontenc}
\usepackage{float}
\usepackage{amssymb,amscd,amsmath,amsfonts,amsthm}
\usepackage{multicol}
\usepackage{tabularx,environ,array}
\usepackage{tikz}
\usepackage{comment}
\usepackage{hyperref}
\usepgflibrary{shapes.geometric}

\hypersetup{colorlinks=true, linkcolor=blue, citecolor=blue, urlcolor=blue}

\newcommand{\imv}{\chi_{\mu_i}}
\newcommand{\mvc}{\chi_{\mu}}

\newtheorem{theorem}{Theorem}

\newtheorem{conjecture}{Conjecture}
\newtheorem{corollary}[theorem]{Corollary}

\newtheorem{lemma}[theorem]{Lemma}

\newtheorem{observation}[theorem]{Observation}
\newtheorem{problem}{Problem}
\newtheorem{proposition}[theorem]{Proposition}
\newtheorem{remark}[theorem]{Remark}

\title{{Independent mutual-visibility coloring and related concepts}}
\date{}
\author{{Bo\v{s}tjan Bre\v{s}ar$^{1,2}$, Iztok Peterin$^{2,3}$, Babak Samadi$^{2}$ and Ismael G. Yero$^{4}$}\vspace{2mm}\\
$^{1}$Faculty of Natural Sciences and Mathematics, University of Maribor, Slovenia\\
$^{2}$Institute of Mathematics, Physics and Mechanics, Ljubljana, Slovenia\\
$^{3}$Faculty of Electrical Engineering and Computer Science, University of Maribor,\\ Maribor, Slovenia\\
$^{4}$Departamento de Matem\'{a}ticas, Universidad de C\'{a}diz, Algeciras Campus, Spain\vspace{1mm}\\
\texttt{bostjan.bresar@um.si}\\
\texttt{iztok.peterin@um.si}\\
\texttt{babak.samadi@imfm.si}\\
\texttt{ismael.gonzalez@uca.es}}
\date{}

\begin{document}

\maketitle

\begin{abstract}
Given a graph $G$, a subset $M\subseteq V(G)$ is a mutual-visibility (MV) set if for every $u,v\in M$, there exists a $u,v$-geodesic whose internal vertices are not in $M$. We investigate proper vertex colorings of graphs whose color classes are mutual-visibility sets. The main concepts that arise in this investigation are independent mutual-visibility (IMV) sets and vertex partitions into these sets (IMV colorings). The IMV number $\mu_{i}$ and the IMV chromatic number $\chi_{\mu_{i}}$ are defined as maximum and minimum cardinality taken over all IMV sets and IMV colorings, respectively. Along the way, we also continue with the study of MV chromatic number $\chi_{\mu}$ (as the smallest number of sets in a vertex partition into MV sets), which was initiated in an earlier paper. 

We establish a close connection between the (I)MV chromatic numbers of subdivisions of complete graphs and Ramsey numbers $R(4^k;2)$. From the computational point of view, we prove that the problems of computing $\chi_{\mu_{i}}$ and $\mu_{i}$ are NP-complete, and that it is NP-hard to decide whether a graph $G$ satisfies $\imv(G)=\alpha(G)$ where $\alpha(G)$ is the independence number of $G$. Several tight bounds on $\chi_{\mu_{i}}$, $\chi_{\mu}$ and $\mu_{i}$ are given. Exact values/formulas for these parameters in some classical families of graphs are proved. In particular, we prove that $\chi_{\mu_{i}}(T)=\chi_{\mu}(T)$ holds for any tree $T$ of order at least $3$, and determine their exact formulas in the case of lexicographic product graphs. Finally, we give tight bounds on the (I)MV chromatic numbers for the Cartesian and strong product graphs, which lead to exact values in some important families of product graphs.
\end{abstract}\vspace{0.5mm}
\textbf{Math.\ Subj.\ Class. 2020:} 05C12, 05C15, 05C55, 05C69, 68Q17\vspace{0.5mm}    


\section{Introduction and preliminaries} 

The notion of visibility between two vertices in a graph was introduced by Di~Stefano in~\cite{DiS} in the following way. Given a graph $G$ and a set $X\subseteq V(G)$, two vertices $x,y\in X$ are {\em visible} with respect to $X$ (or $X$-{\em visible} as renamed in \cite{CDDH})  if there exists an $x,y$-geodesic (that is, a shortest path between $x$ and $y$) such that all its internal vertices are not in $X$. In this sense, a set $X$ is a {\em mutual-visibility set} (MV set) in $G$ if every two vertices in $X$ are $X$-visible. The maximum cardinality among all MV sets in $G$ is the {\em mutual-visibility number} of $G$, denoted $\mu(G)$. 

Several motivations for studying mutual-visibility in graphs were given in the seminal paper~\cite{DiS}, which are mainly arising in some computer science related models. In particular, the concept of mutual-visibility has been intensively studied in the context of sets of points in the Euclidean plane modeling the problem of placing mobile robots in such a way that they are seen among each other through some geodesic. A number of papers continued the investigation from~\cite{DiS} considering various aspects of mutual-visibility~\cite{BY, CDDH,CDK,CDKY, CDKY2,KV, Kuziak, TK}. It is a remarkable fact that the mutual-visibility number of some graphs is closely related to several classical combinatorial topics like the Zarankiewicz problem or some Tur\'an-type problems; see for instance \cite{BKT,CDK,CDKY2}. Moreover, several variations of the original concept have rapidly appeared. In particular, a variety of them is discussed in \cite{CDDH}.

The independent version of the mutual-visibility notion is known from \cite{CDK}, where the concept naturally appeared while the authors were studying the mutual-visibility number in Cartesian products of graphs. One can simply say that this new notion appeared by imposing the independence property on MV sets. That is, a set $X$ is an {\em independent mutual-visibility set} (IMV set for short) if $X$ is independent and every two vertices in $X$ are $X$-visible.
The cardinality of a largest IMV set in $G$ is the \textit{independent mutual-visibility number} (\textit{IMV number} for short), denoted by $\mu_i(G)$.

On the other hand, in the recent manuscript~\cite{KKVY}, the authors considered a vertex coloring of a graph with respect to the mutual-visibility property. Setting $[k]=\{1,\ldots,k\}$, a function $c:V(G)\rightarrow [k]$ is a {\em mutual-visibility $k$-coloring} (MV $k$-coloring) of a graph $G$ if every color class $c^{-1}(i)$, with $i\in [k]$, is an MV set in $G$. The smallest integer $k$ such that $G$ admits an MV $k$-coloring is the {\em mutual-visibility chromatic number} (MV chromatic number) of $G$, denoted $\mvc(G)$. We remark that some further contributions on this new parameter for the case of hypercubes were immediately after \cite{KKVY} presented in \cite{Axenovich}.

In connection with the two notions above, a natural continuation of these mutual-visibility investigations is to consider independent mutual-visibility colorings of graphs. Formally, a function $c:V(G)\rightarrow [k]$ is an {\em independent mutual-visibility $k$-coloring} (IMV $k$-coloring) of a graph $G$ if every color class $c^{-1}(i)$, with $i\in [k]$, is an IMV set in $G$. The smallest integer $k$ such that $G$ admits an IMV $k$-coloring is the {\em independent mutual-visibility chromatic number} (IMV chromatic number) of $G$, denoted $\imv(G)$. Consider for instance the graph $G^*$ of Figure \ref{fig:K_4-subdiv}. One can verify that each set formed by the vertices with the same color and shape forms an IMV set, and that the resulting coloring is an optimal one. Thus, $\imv(G^*)=3$.

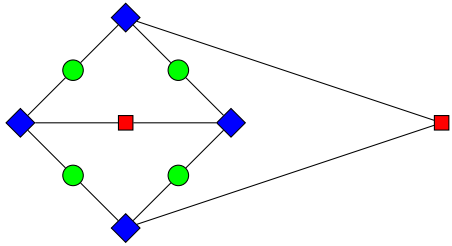
\begin{figure}[ht]
\centering
\begin{tikzpicture}[scale=.7, transform shape]
\node [draw, shape=rectangle, fill=red] (a1) at  (0,0) {};
\node [draw, shape=rectangle, fill=red] (a2) at  (6,0) {};

\node [draw, shape=diamond, fill=blue] (b1) at  (0,-2) {};
\node [draw, shape=diamond, fill=blue] (b2) at  (0,2) {};
\node [draw, shape=diamond, fill=blue] (b3) at  (-2,0) {};
\node [draw, shape=diamond, fill=blue] (b4) at  (2,0) {};

\node [draw, shape=circle, fill=green] (c1) at  (-1,-1) {};
\node [draw, shape=circle, fill=green] (c2) at  (-1,1) {};
\node [draw, shape=circle, fill=green] (c3) at  (1,-1) {};
\node [draw, shape=circle, fill=green] (c4) at  (1,1) {};

\draw(b3)--(a1)--(b4);
\draw(b1)--(c3)--(b4)--(c4)--(b2)--(c2)--(b3)--(c1)--(b1)--(a2)--(b2);

\end{tikzpicture}
\caption{The graph $G^*$, which is the complete graph $K_4$ with subdivided edges. An optimal IMV $3$-coloring is drawn.}
\label{fig:K_4-subdiv}
\end{figure}

We also remark that considering the independent version of mutual-visibility in the frame of vertex colorings of graphs might be more natural than its predecessor from~\cite{KKVY}, since an independent mutual-visibility coloring is also a proper coloring of the graph. Hence, this parameter belongs to the extensive list of graph invariants, which arise from the chromatic number by imposing to it some additional constraint(s).

The main goal of this paper is to consider the (I)MV chromatic number of graphs. We continue the investigation on $\mvc(G)$, as well as, initiate the study of $\imv(G)$. Along the way, we also present some results regarding the IMV number, giving more insight into the dynamic area of mutual-visibility investigations.


\subsection{Plan of the article}

We end this section with some extra terminology and notations that we need, together with a few basic observations. Section \ref{sec:subdivision} initiates the exposition of our results, where we first consider the subdivision graphs of complete graphs and complete bipartite graphs. This study already shows the difficulty of the (I)MV coloring problem. Specifically, we show that finding the (I)MV chromatic number of the subdivision graph of complete (bipartite) graphs relates to classical problems in Ramsey theory (notice that the example of Figure \ref{fig:K_4-subdiv} already represents a subdivision of the complete graph $K_4$, for which the (I)MV chromatic number equals $3$). 

In Section \ref{sec:complex}, we show that the decision problems associated with $\mu_{i}$ and $\chi_{\mu_{i}}$ are NP-complete. In particular, we prove that deciding whether equality happens in the inequality $\mu_{i}(G)\leq \alpha(G)$ (arisen from the definitions), in which $\alpha$ stands for the independence number, is NP-hard. 

Further on, tight bounds on the (I)MV chromatic numbers are given in Section~\ref{BE}. In particular, we prove that these two parameters are equal for trees on at least three vertices, and give the exact values/formulas for $\chi_{\mu}$, $\chi_{\mu_{i}}$ and $\mu_{i}$ in the case of lexicographic product graphs. We also consider graphs with diameter $2$ that were already studied in~\cite{KKVY}, and improve the bound on $\mvc(G)$ for graphs $G$ with ${\rm diam}(G)=2$ using the $1$-defective chromatic number of $G$. 

Finally, the (I)MV chromatic numbers are investigated in Section \ref{Cart-Str} in connection with the Cartesian and strong products, where we give general bounds, as well as, exact formulas when some factor is a special graph. In particular, we obtain the exact values of $\imv(P_t\boxtimes P_{4k})$ for all positive integers $t$ and $k$, and we determine the MV chromatic numbers for all strong products of paths. 


\subsection{Basic observations, terminology and notation}

The order $|V(G)|$ of a graph $G$ will be denoted by $n(G)$. For a vertex $v\in V(G)$, the \textit{open neighborhood} of $v$, denoted $N_G(v)$, is the set of vertices adjacent to $v$. If $S\subseteq V(G)$, then the \textit{open neighborhood} of $S$ is $N_G(S)=\left(\bigcup_{v\in S} N_G(v)\right)\setminus S$. Moreover, the subgraph induced by the set $V(G)\setminus S$ will be represented as $G-S$. 

A \textit{proper $k$-coloring} of a graph $G$ is a function $c:V(G)\rightarrow[k]$ such that for any two adjacent vertices $x,y\in V(G)$, $c(x)\ne c(y)$ holds. The \textit{chromatic number} $\chi(G)$ of $G$ is the smallest $k$ such that $G$ admits a proper $k$-coloring. 
Clearly, any IMV $k$-coloring of $G$ is a proper coloring of $G$ as well. On the other hand, it is clear that any IMV $k$-coloring is also an MV $k$-coloring. In addition, since each color class of any IMV $k$-coloring is an IMV set, we deduce that $\imv(G)\ge \left\lceil n(G)/\mu_i(G)\right\rceil$. Altogether, the following observation follows for any graph $G$:
\begin{equation}
\label{eq:first-ineq}
\imv(G)\ge \max\left\{\chi(G),\mvc(G),\left\lceil\frac{n}{\mu_i(G)}\right\rceil\right\}.
\end{equation}

A subgraph $H$ of a graph $G$ is called \textit{convex} if all geodesics in $G$ between vertices of $H$ entirely lie in $H$. The first claim of the following lemma was implicitly used in~\cite{KKVY}, while the second claim is similarly easy. We provide the proof of this auxiliary result for completeness.

\begin{lemma}\label{lem:convex}
If $H$ is a convex subgraph in $G$, then $\chi_{\mu}(G)\ge \chi_{\mu}(H)$ and $\imv(G)\ge \imv(H)$.
\end{lemma}
\begin{proof}
Let ${\cal S}=\{S_1,\ldots, S_{\chi_\mu(G)}\}$ be a $\chi_\mu(G)$-partition. Consider the restriction ${\cal S}_H$ of $\cal S$ to $V(H)$. Let $x$ and $y$ be any two vertices in $H$ that both belong to the same set from $\cal S$, say $x,y\in S_i$. Since $\cal S$ is a $\chi_\mu(G)$-partition of $G$, there exists an $x,y$-geodesic $P$ in $G$ such that no internal vertices on $P$ belong to $S_i$. Now, since $H$ is convex in $G$, all vertices of $P$ lie in $H$. In other words, $P$ is an $x,y$-geodesic in $H$ with no internal vertices in $S_{i}\cap V(H)$. Therefore, ${\cal S}_H$ is a vertex partition of $H$ into MV sets. Thus, $\chi_{\mu}(H)\leq|{\cal S}_H|\le|{\cal S}|={\chi_\mu}(G)$.

The second inequality follows from the argument above by replacing a ``$\chi_\mu(G)$-partition'' with a ``$\imv(G)$-partition''.
\end{proof}

If $G$ and $H$ are two graphs, then the vertex set of the four standard products of $G$ and $H$ is $V(G)\times V(H)$; see~\cite{HIK}; while their edge sets are defined as follows.
\begin{itemize}
\item In the \emph{Cartesian product} $G\square H$, two vertices are adjacent if they are adjacent in one coordinate and equal in the other.
\item In the \emph{direct product} $G\times H$, two vertices are adjacent if they are adjacent in both coordinates.
\item The edge set of the \emph{strong product} $G\boxtimes H$ is the union of $E(G\square H)$ and $E(G\times H)$.
\item Two vertices $(g,h)$ and $(g',h')$ are adjacent in the \emph{lexicographic product} $G\circ H$ if either $gg'\in E(G)$ or $g=g'$ and $hh'\in E(H)$.
\end{itemize}

The edges of the strong product that belong to $E(G\square H)$ are \emph{Cartesian edges} and the other are \emph{non-Cartesian edges}. The subgraph of any product induced on the set $^{g}\! H=\{(g,h):h\in V(H)\}$ for some $g\in V(G)$ is an $H$-\emph{fiber}, and is isomorphic to $H$ in the Cartesian, the strong and the lexicographic product. Similarly, for an arbitrary $h\in V(H)$ we consider the set $G^h$ obtaining a $G$-\emph{fiber}.   

Throughout the paper, by an $\eta(G)$-coloring and a $\theta(G)$-set, with $\eta\in \{\chi,\chi_{\mu},\chi_{\mu_{i}}\}$ and $\theta\in \{\alpha,\mu_{i}\}$, we mean a coloring, an MV coloring, an IMV coloring, an independent set and an IMV set of $G$ of cardinality $\chi(G)$, $\chi_{\mu}(G)$, $\chi_{\mu_{i}}$, $\alpha(G)$ and $\mu_{i}(G)$, respectively. We use \cite{we} as a reference for terminology and notation that are not explicitly defined here.


\section{Subdivision of complete graphs and Ramsey numbers}\label{sec:subdivision}

Given a graph $G$, the \textit{subdivision} of an edge $e=uv$ of $G$ is the operation made on $G$ so that we remove the edge $e=uv$, add a new vertex $e_{uv}$ and the two edges $ue_{uv}$ and $ve_{uv}$. Such a vertex $e_{uv}$ is called a \textit{subdivision vertex}. The \textit{subdivision graph} $S(G)$ of $G$ is obtained from $G$ by subdividing all its edges. We next consider the IMV chromatic number of the subdivision graphs of $K_n$ and $K_{r,s}$. To this end, we need the following terminology. Given a graph $H$, a graph $G$ is \textit{$H$-free} if no subgraph of $G$ is isomorphic to $H$. If $F$ is a subset of edges of $G$, then the \textit{edge-induced subgraph} $\langle F\rangle$ is the subgraph of $G$ formed by the edges of $F$ together with all their endvertices. By an \textit{$H$-free partition} of $G$, we mean a partition $P=\{P_1,\dots,P_r\}$ of $E(G)$ in such a way that each edge-induced subgraph $\langle P_i\rangle$ is an $H$-free graph. 

We next consider the case when $G$ is the complete graph $K_n$ and $H$ is the complete graph $K_4$. In this sense, let $\rho(n)$ be the smallest number of sets among all $K_4$-free partitions of $K_n$.

\begin{theorem}\label{th:subdiv-K_n}
For any positive integer $n$, $\rho(n)\le\mvc \big{(}S(K_n)\big{)}\le \imv\big{(}S(K_n)\big{)}\le \rho(n)+1$.
\end{theorem}
\begin{proof}
Assume $U=V(K_n)=[n]$. According to the definition of $S(K_n)$, it is readily observed that $U$ is an IMV set of $S(K_n)$ as any two vertices $i,j\in U$ are not adjacent in $S(K_n)$ and are $U$-visible through the geodesic $ie_{ij}j$. Now, consider a $K_4$-free partition $P=\{P_1,\dots,P_r\}$ of $K_n$, and let $P'=\{P'_1,\dots,P'_r\}$ be the collection of vertex subsets in $S(K_n)$ such that for each $k\in[r]$ it holds $e_{ij}\in P_{k}'$ if and only if $ij\in P_{k}$. 

 We claim that each $P'_k$ is an IMV set of $S(K_n)$. Let $e_{ij},e_{rs}\in P'_k$. If the distinct edges $ij,rs\in E(K_n)$ have a common endvertex, say $j=r$, then $e_{ij}je_{rs}$ is an $e_{ij},e_{rs}$-geodesic in $S(K_n)$. Thus, $e_{ij}$ and $e_{rs}$ are $P'_k$-visible. On the other hand, if $ij$ and $rs$ have no common endvertex in $K_n$, then at least one of the edges $ir$, $is$, $jr$ or $js$, say $ir$, does not belong to $P'_k$ since the subgraph $\langle P_k\rangle$ is $K_4$-free. Thus, $e_{ij}ie_{ir}re_{rs}$ is an $e_{ij},e_{rs}$-geodesic in $S(K_n)$, which does not contain any internal vertex from $P'_k$. So, $e_{ij}$ and $e_{rs}$ are again $P'_k$-visible. 

As a consequence of the arguments above, we deduce that $P'\cup \{U\}$ is an IMV coloring of $S(K_n)$. Therefore, $\imv \big{(}S(K_n)\big{)}\le \rho(n)+1$.

To prove that $\mvc \big{(}S(K_n)\big{)}\ge \rho(n)$, we consider an MV coloring $Q=\{Q_1,\dots,Q_t\}$ of $S(K_n)$ with $t=\mvc \big{(}S(K_n)\big{)}$. Now, let $Q'=\{Q'_1,\dots,Q'_t\}$ be the collection of edge subsets of $K_n$ such that an edge $ij\in Q'_k$ if and only if $e_{ij}\in Q_k$. Clearly, there might be some sets $Q'_\ell$ in $Q'$ that are empty and let $q$ be the number of nonempty subsets of $Q'$. Since each $Q_k$ is an MV set of $S(K_n)$, in order that its vertices will be $Q_k$-visible in $S(K_n)$, it must happen that either $Q'_k$ is empty or the edge-induced subgraph $\langle Q'_k\rangle$ is $K_4$-free. Hence, the collection of all nonempty subsets of $Q'$ forms a $K_4$-free partition of $K_n$ (notice that each edge of $K_n$ belongs to one set of $Q'$ due to the definition of $Q'$). Thus, $\rho(n)\le q\leq \mvc \big{(}S(K_n)\big{)}$. 

Finally, the inequality chain is completed by using \eqref{eq:first-ineq}.
\end{proof}

There is a close connection between the (I)MV chromatic numbers of subdivisions of complete graphs and some classical notions from Ramsey theory. Denote by $R(m^k)$ the Ramsey number $$R(\underbrace{m,\ldots,m}_\text{\scriptsize $k$}),$$ which stands for the smallest integer $N$ such that every $k$-coloring of $E(K_N)$ contains a subgraph isomorphic to $K_m$ all edges of which are colored with the same color from $[k]$.

Fixing $m$, note that $R(m^k)$ is an increasing function of $k$, that is, $R(m^k)<R(m^{k+1})$ holds for any positive integer $k$. For any positive integer $n$, the definition of the function $\rho$ yields $$R(4^{\rho(n)-1})\le n < R(4^{\rho(n)}),$$ which in turn implies the following remark. 

\begin{observation}
\label{obs:Ramsey}
For a positive integer $n$, $\rho(n)$ represents the smallest integer $k$ such that $R(4^k)>n$.
\end{observation}

It is well known that $R(4,4)=18$, but for $k\ge 3$, determining $R(4^k)$ is a formidable problem. Therefore, we strongly suspect that determining $\imv(G)$ and $\mvc(G)$ is also a (computationally) difficult problem since its solution would lead to very close approximations of some Ramsey numbers. In particular, in the next section, we establish the NP-completeness of the decision version of the IMV chromatic number. 

Now consider the case when $G$ is the complete bipartite graph $K_{r,s}$ and $H$ is the cycle $C_4$ in the definition of $H$-free partitions. Hence, let $\rho(r,s)$ be the smallest number of sets among all $C_4$-free partitions of $K_{r,s}$. The following result can be deduced, by using similar arguments as the ones from Theorem \ref{th:subdiv-K_n}, only taken into account that we need to consider $C_4$-free partitions instead of $K_4$-free partitions.  

\begin{theorem}\label{th:subdiv-K_r-s}
For any positive integers $r,s$, $\rho(r,s)\le\mvc \big{(}S(K_{r,s})\big{)}\le \imv\big{(}S(K_{r,s})\big{)}\le \rho(r,s)+1$.
\end{theorem}

In a similar fashion, as for $\rho(n)$, we remark that if $r=s$, then indeed the value of $\rho(s,s)$ represents the smallest integer $k$ such that $br\big{(}(C_4)^k\big{)}>s$, where $br\big{(}(C_4)^k\big{)}$ is the bipartite Ramsey number $br(\underbrace{C_4,\ldots,C_4}_\text{\scriptsize $k$})$. Bipartite Ramsey numbers have been widely studied, and for the specific case of bipartite Ramsey numbers of cycles, the reader can find a few good contributions when $k=2$ or $k=3$ (see for instance \cite{Yan-24} and references therein). Determining these numbers for larger values of $k$ is widely open.


\section{Complexity issues}\label{sec:complex}

In this section, we first determine the computational complexity of computing the IMV number and the IMV chromatic number. More precisely, we analyze the following decision problems.  
$$\begin{tabular}{|l|}
\hline
\mbox{\sc Independent Mutual-Visibility}\\
\mbox{{\sc Instance}: A graph $G$ and a positive integer $r\leq n(G)$.}\\
\mbox{{\sc Question}: Does there exist an IMV set in $G$ of cardinality at least $r$?}\\
\hline
\end{tabular}$$ 
and
$$\begin{tabular}{|l|}
\hline
\mbox{\sc Independent Mutual-Visibility Coloring}\\
\mbox{{\sc Instance}: A graph $G$ and a positive integer $r\leq n(G)$.}\\
\mbox{{\sc Question}: Is there an IMV coloring in $G$ with at most $r$ colors?}\\
\hline
\end{tabular}$$

For this purpose, we use the structure of the corona product of graphs defined as follows. Let $G$ and $H$ be graphs where $V(G)=\{v_1,\ldots,v_{n}\}$. The \textit{corona product} $G\odot H$ of $G$ and $H$ is obtained from the disjoint union of $G$ and $n$ disjoint copies of $H$, say $H_1,\ldots,H_{n}$, such that $v_j\in V(G)$ is adjacent to all vertices of $H_j$ for each $j\in[n]$.

We make use of the following decision problem, which is well known to be NP-complete (see \cite{GJ}), to prove the next result.

$$\begin{tabular}{|l|}
\hline
\mbox{\sc Independent Set}\\
\mbox{{\sc Instance}: A graph $G$ and a positive integer $k\leq n(G)$.}\\
\mbox{{\sc Question}: Does $G$ have an independent set of cardinality at least $k$?}\\
\hline
\end{tabular}$$ 

\begin{theorem}\label{Complexity1}
{\sc Independent Mutual-Visibility} is NP-complete even for graphs with universal vertices.
\end{theorem}
\begin{proof}
The problem belongs to NP since checking that a given subset $P$ of vertices is both an independent set and an MV set of cardinality at least $r$ can be done in polynomial time (regarding mutual visibility, this was specifically checked by an $O\big{(}|P|(n+m)\big{)}$-time algorithm in \cite{DiS}, in which $n$ and $m$ are the order and the size of the graph, respectively). 

To prove that the problem is NP-hard, we describe a polynomial transformation from the {\sc Independent Set} to our problem. Consider a graph $H$ and a positive integer $r\leq n(H)$ as an instance of the {\sc Independent Set}. We consider $G\odot H$, where $G$ is any graph. Let $J_{j}$ be an $\alpha(H_{j})$-set for each $j\in[n(G)]$. Note that $J=\bigcup_{j=1}^{n(G)}J_{j}$ is simultaneously an MV set in $G\odot H$ and an $\alpha(G\odot H)$-set. Therefore, $|J|\leq \mu_{i}(G\odot H)\leq \alpha(G\odot H)=|J|$. Hence,
\begin{equation}\label{Cor}
\mu_{i}(G\odot H)=n(G)\alpha(H)
\end{equation} 
for any graphs $G$ and $H$. In particular, we have $\mu_{i}(K_{1}\odot H)=\alpha(H)$. Our reduction is now completed by taking into account that $\mu_{i}(K_{1}\odot H)\geq r$ if and only if $\alpha(H)\geq r$. Since {\sc Independent Set} is NP-complete (and since $K_{1}\odot H$ has a universal vertex), {\sc Independent Mutual-Visibility} is NP-complete even for graphs with universal vertices.
\end{proof}

The following famous decision problem, which is known to be NP-complete (see \cite{GJ}), turns out to be useful in the proof of the next theorem.

$$\begin{tabular}{|l|}
\hline
\mbox{\sc Graph colorability}\\
\mbox{{\sc Instance}: A graph $G$ and a positive integer $k\leq n(G)$.}\\
\mbox{{\sc Question}: Does there exist a proper coloring of $G$ with at most $k$ colors?}\\
\hline
\end{tabular}$$

\begin{theorem}\label{Complexity2}
{\sc Independent Mutual-Visibility Coloring} is NP-complete even for graphs with universal vertices.
\end{theorem}
\begin{proof}
The problem is in NP as verifying that a given vertex partition of $G$ is indeed an IMV coloring of cardinality at most $k$ can be done in polynomial time.

Consider a graph $G$ and a positive integer $k\leq n(G)$ as an instance of {\sc Graph Colorability}. We set $G'=K_{1}\odot G$ and $r=k+1$. Note that the IMV sets in $G'$ are precisely the independent sets in $G'$. Hence, $\chi_{\mu_{i}}(G')=\chi(G')=\chi(G)+1$. Consequently, because $\chi_{\mu_{i}}(G')\leq r$ if and only if $\chi(G)\leq k$, and since {\sc Graph Colorability} is NP-complete and $G'$ has a universal vertex, we deduce the statement of the theorem. 
\end{proof}

Cicerone et al.~\cite{CDK} proved the following. (In fact, they proved that if diam$(G)\leq3$, then every independent set in $G$ is also an MV set in $G$.)

\begin{proposition}\label{Pro}
Let $G$ be a graph. If \emph{diam}$(G)\leq3$, then $\mu_{i}(G)=\alpha(G)$.
\end{proposition}

Proposition \ref{Pro} does not necessarily hold for graphs with diam$(G)\geq4$. For instance, $\mu_{i}(P_{5})=2$ and $\alpha(P_{5})=3$ (\cite{CDK}). 

In view of the clear inequality $\mu_{i}(G)\leq \alpha(G)$ for any graph $G$ and Proposition \ref{Pro}, one might hope to practically characterize the extremal graphs for this inequality. The following result shows that this hope is in vain (even for graphs with diameter $4$).

\begin{theorem}\label{Hard}
It is NP-hard to decide whether $\mu_{i}(G)=\alpha(G)$.
\end{theorem}
\begin{proof}
We describe a polynomial transformation from the well-known $3$-SAT problem to our problem. Consider an arbitrary instance of the $3$-SAT problem, given by $U=\{u_{1},\ldots,u_{a}\}$ (with the set of complements $U'=\{u_{1}',\ldots,u_{a}'\}$) and a collection $C=\{C_{1},\ldots,C_{b}\}$ of $3$-variable clauses over $U\cup U'$. We consider:\vspace{-2mm}
\begin{itemize}
\item[$\bullet$] a path $c_{j}c_{j}'$ associated with every clause $C_{j}$, and join a new vertex $c$ to all vertices in $\{c_{1}',\ldots,c_{b}'\}$,\vspace{-2mm}
\item[$\bullet$] a tree $T_{i}$, with $V(T_{i})=\{p_{i},p_{i}',q_{i},r_{i},s_{i}\}$ and $E(T_{i})=\{p_{i}p_{i}',p_{i}'q_{i},q_{i}r_{i},q_{i}s_{i}\}$, associated with every variable $u_{i}$, and\vspace{-2mm}
\item[$\bullet$] a triangle $xyzx$.\vspace{-2mm}
\end{itemize}
Finally, we join by an edge:\vspace{-2mm}
\begin{itemize}
\item[$\bullet$] $p_{i}$ to $c_{j}$ (resp. $p_{i}'$ to $c_{j}$) if $u_{i}\in C_{j}$ (resp. $u_{i}'\in C_{j}$) for each $i\in[a]$ and $j\in[b]$,\vspace{-2mm}
\item[$\bullet$] $x$ to $p_{i}$ and $p_{i}'$ for each $i\in[a]$,\vspace{-2mm}
\item[$\bullet$] $y$ to $r_{i}$ and $s_{i}$ for each $i\in[a]$, and\vspace{-2mm}
\item[$\bullet$] $c$ to $r_{i}$ and $s_{i}$ for each $i\in[a]$.\vspace{-2mm}
\end{itemize}
Let $G$ be the resulting graph (see Figure \ref{Fig1} for an example). It is clear that the construction of $G$ is accomplished in polynomial time and that $G$ is of diameter $4$. The following claim turns out to be useful in the rest of the proof.

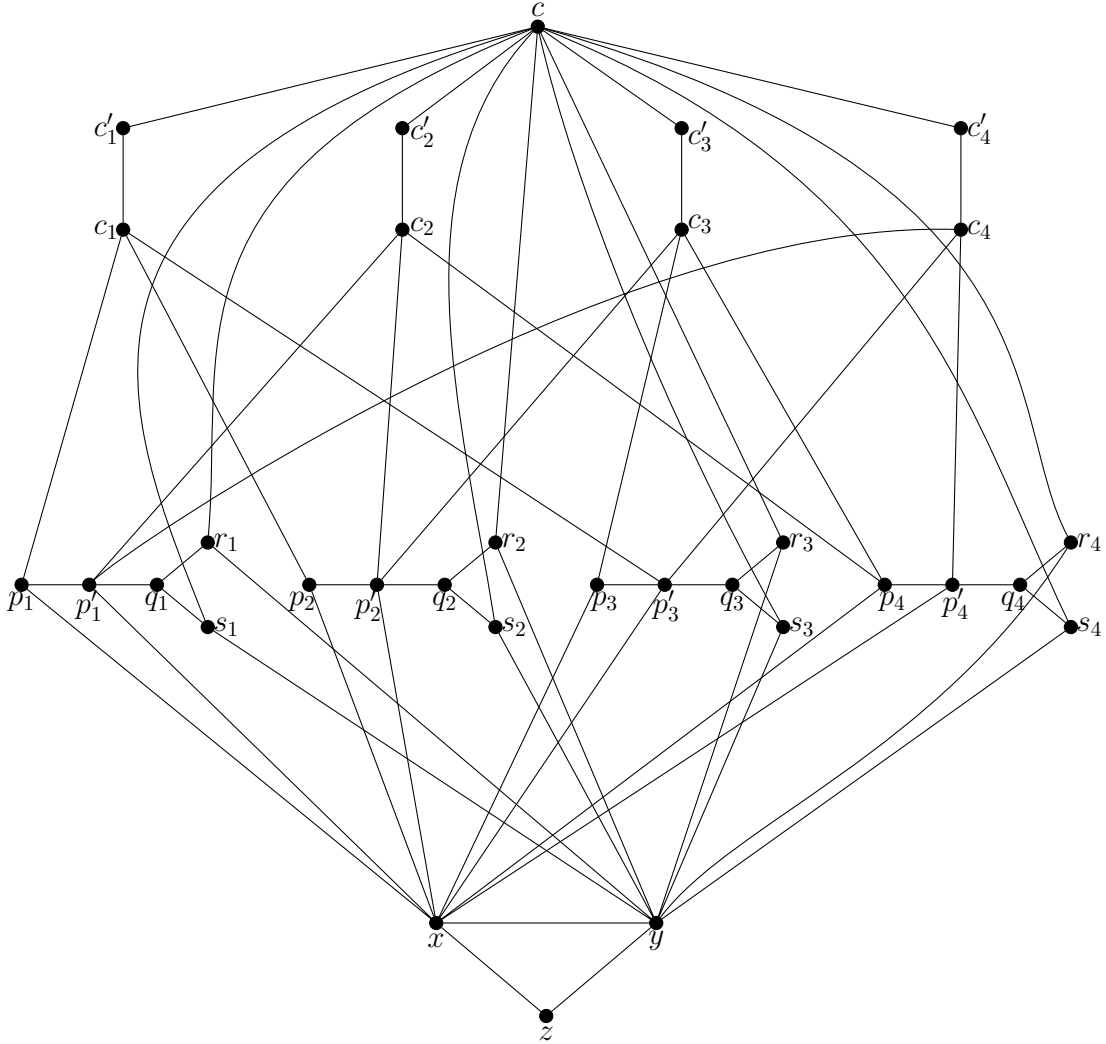
\begin{figure}[ht!]
\centering
\begin{tikzpicture}[scale=0.45, transform shape]
\node [draw, shape=circle, fill=black] (p_{1}) at (0,0) {};
\node [draw, shape=circle, fill=black] (p_{1}') at (2,0) {};
\node [draw, shape=circle, fill=black] (q_{1}) at (4,0) {};
\node [draw, shape=circle, fill=black] (r_{1}) at (5.5,1.25) {};
\node [draw, shape=circle, fill=black] (s_{1}) at (5.5,-1.25) {};
\draw (p_{1})--(p_{1}')--(q_{1});
\draw (r_{1})--(q_{1})--(s_{1});
\node [scale=1.8] at (0,-0.5) {\large $p_{1}$};
\node [scale=1.8] at (2,-0.6) {\large $p_{1}'$};
\node [scale=1.8] at (4,-0.5) {\large $q_{1}$};
\node [scale=1.8] at (6.05,1.25) {\large $r_{1}$};
\node [scale=1.8] at (6.05,-1.25) {\large $s_{1}$};

\node [draw, shape=circle, fill=black] (p_{2}) at (8.5,0) {};
\node [draw, shape=circle, fill=black] (p_{2}') at (10.5,0) {};
\node [draw, shape=circle, fill=black] (q_{2}) at (12.5,0) {};
\node [draw, shape=circle, fill=black] (r_{2}) at (14,1.25) {};
\node [draw, shape=circle, fill=black] (s_{2}) at (14,-1.25) {};
\draw (p_{2})--(p_{2}')--(q_{2});
\draw (r_{2})--(q_{2})--(s_{2});
\node [scale=1.8] at (8.3,-0.5) {\large $p_{2}$};
\node [scale=1.8] at (10.25,-0.6) {\large $p_{2}'$};
\node [scale=1.8] at (12.5,-0.5) {\large $q_{2}$};
\node [scale=1.8] at (14.55,1.25) {\large $r_{2}$};
\node [scale=1.8] at (14.55,-1.25) {\large $s_{2}$};

\node [draw, shape=circle, fill=black] (p_{3}) at (17,0) {};
\node [draw, shape=circle, fill=black] (p_{3}') at (19,0) {};
\node [draw, shape=circle, fill=black] (q_{3}) at (21,0) {};
\node [draw, shape=circle, fill=black] (r_{3}) at (22.5,1.25) {};
\node [draw, shape=circle, fill=black] (s_{3}) at (22.5,-1.25) {};
\draw (p_{3})--(p_{3}')--(q_{3});
\draw (r_{3})--(q_{3})--(s_{3});
\node [scale=1.8] at (17.23,-0.5) {\large $p_{3}$};
\node [scale=1.8] at (19.05,-0.6) {\large $p_{3}'$};
\node [scale=1.8] at (21,-0.5) {\large $q_{3}$};
\node [scale=1.8] at (23.05,1.25) {\large $r_{3}$};
\node [scale=1.8] at (23.05,-1.25) {\large $s_{3}$};

\node [draw, shape=circle, fill=black] (p_{4}) at (25.5,0) {};
\node [draw, shape=circle, fill=black] (p_{4}') at (27.5,0) {};
\node [draw, shape=circle, fill=black] (q_{4}) at (29.5,0) {};
\node [draw, shape=circle, fill=black] (r_{4}) at (31,1.25) {};
\node [draw, shape=circle, fill=black] (s_{4}) at (31,-1.25) {};
\draw (p_{4})--(p_{4}')--(q_{4});
\draw (r_{4})--(q_{4})--(s_{4});
\node [scale=1.8] at (25.73,-0.5) {\large $p_{4}$};
\node [scale=1.8] at (27.6,-0.55) {\large $p_{4}'$};
\node [scale=1.8] at (29.3,-0.5) {\large $q_{4}$};
\node [scale=1.8] at (31.55,1.25) {\large $r_{4}$};
\node [scale=1.8] at (31.55,-1.25) {\large $s_{4}$};

\node [draw, shape=circle, fill=black] (z) at (15.5,-12.75) {};
\node [draw, shape=circle, fill=black] (x) at (12.25,-10) {};
\node [draw, shape=circle, fill=black] (y) at (18.75,-10) {};
\draw (x)--(y)--(z)--(x);
\node [scale=1.8] at (12.25,-10.5) {\large $x$};
\node [scale=1.8] at (18.75,-10.5) {\large $y$};
\node [scale=1.8] at (15.5,-13.25) {\large $z$};

\draw (p_{1})--(x)--(p_{1}');
\draw (p_{2})--(x)--(p_{2}');
\draw (p_{3})--(x)--(p_{3}');
\draw (p_{4})--(x)--(p_{4}');

\draw (r_{1})--(y)--(s_{1});
\draw (r_{2})--(y)--(s_{2});
\draw (r_{3})--(y)--(s_{3});
\draw (s_{4})--(y);
\draw (y) .. controls (21,-7) and (27,-5) .. (r_{4});

\node [draw, shape=circle, fill=black] (c_{1}') at (3,13.5) {};
\node [draw, shape=circle, fill=black] (c_{1}) at (3,10.5) {};
\draw (c_{1})--(c_{1}');
\node [scale=1.8] at (2.5,13.5) {\large $c_{1}'$};
\node [scale=1.8] at (2.5,10.5) {\large $c_{1}$};

\node [draw, shape=circle, fill=black] (c_{2}') at (11.25,13.5) {};
\node [draw, shape=circle, fill=black] (c_{2}) at (11.25,10.5) {};
\draw (c_{2})--(c_{2}');
\node [scale=1.8] at (11.85,13.4) {\large $c_{2}'$};
\node [scale=1.8] at (11.85,10.6) {\large $c_{2}$};

\node [draw, shape=circle, fill=black] (c_{3}') at (19.5,13.5) {};
\node [draw, shape=circle, fill=black] (c_{3}) at (19.5,10.5) {};
\draw (c_{3})--(c_{3}');
\node [scale=1.8] at (20.05,13.27) {\large $c_{3}'$};
\node [scale=1.8] at (20.05,10.68) {\large $c_{3}$};

\node [draw, shape=circle, fill=black] (c_{4}') at (27.75,13.5) {};
\node [draw, shape=circle, fill=black] (c_{4}) at (27.75,10.5) {};
\draw (c_{4})--(c_{4}');
\node [scale=1.8] at (28.3,13.5) {\large $c_{4}'$};
\node [scale=1.8] at (28.3,10.5) {\large $c_{4}$};

\node [draw, shape=circle, fill=black] (c) at (15.25,16.5) {};
\draw (c_{1}')--(c)--(c_{2}');
\draw (c_{3}')--(c)--(c_{4}');
\node [scale=1.8] at (15.25,17) {\large $c$};

\draw (p_{1})--(c_{1})--(p_{2});
\draw (c_{1})--(p_{3}');

\draw (p_{1}')--(c_{2})--(p_{2}');
\draw (c_{2})--(p_{4});

\draw (p_{2}')--(c_{3})--(p_{3});
\draw (c_{3})--(p_{4});

\draw (c_{4})--(p_{3}');
\draw (c_{4})--(p_{4}');
\draw (p_{1}') .. controls (3,1) and (17.3,10.8) .. (c_{4});

\draw (c) .. controls (4,12) and (6,7) .. (r_{1});
\draw (c) .. controls (0,12) and (3,5) .. (s_{1});

\draw (c) .. controls (11,12) and (13,5) .. (s_{2});
\draw (c)--(r_{2});

\draw (c)--(r_{3});
\draw (c) .. controls (16.25,12) and (19,5) .. (s_{3});

\draw (c) .. controls (31,12) and (29,5) .. (r_{4});
\draw (c) .. controls (26.5,12) and (28,5) .. (s_{4});

\end{tikzpicture}
\caption{{\small An illustration of the graph $G$, constructed in the proof of Theorem \ref{Hard}, with $a=b=4$. Here, $U=\{u_{1},u_{2},u_{3},u_{4}\}$, $C_{1}=\{u_{1},u_{2},u_{3}'\}$, $C_{2}=\{u_{1}',u_{2}',u_{4}\}$, $C_{3}=\{u_{2}',u_{3},u_{4}\}$ and $C_{4}=\{u_{1}',u_{3}',u_{4}'\}$. Notice that $\big{(}f(u_{1}),f(u_{2}),f(u_{3}),f(u_{4})\big{)}=\big{(}True,False,True,False\big{)}$ is a satisfying truth assignment for $C$ and $\bigcup_{i=1}^{4}\{c_{i}',r_{i},s_{i}\}\cup \{z\}\cup \{p_{1}',p_{2},p_{3}',p_{4}\}$ is an IMV set in $G$ of cardinality $\alpha(G)=3a+b+1=17$.}}
\label{Fig1}
\end{figure}

\medskip
\noindent\textit{Claim A.} $\alpha(G)=3a+b+1$.\vspace{0.75mm}\\
\textit{Proof of Claim A.} We observe that $J=\{c_{1}',\ldots,c_{b}'\}\cup(\bigcup_{i=1}^{a}\{r_{i},s_{i},p_{i}\})\cup \{z\}$ is an independent set in $G$. Therefore, $\alpha(G)\geq|J|=3a+b+1$. On the other hand, if $I$ is an $\alpha(G)$-set, it necessarily follows that $I$ has at most $(i)$ $b$ vertices from $\{c_{1},\ldots,c_{b},c_{1}',\ldots,c_{b}'\}$, $(ii)$ $3$ vertices from each subgraph $T_{i}$, and $(iii)$ $1$ vertex from the triangle $xyzx$. Moreover, if $c\in I$, then no vertices from $\bigcup_{i=1}^{a}\{r_{i},s_{i}\}$ belong to $I$. In such a situation,
\begin{center}
$\alpha(G)=|I|=1+|I\cap \{c,c_{1},\ldots,c_{b},c_{1}',\ldots,c_{b}'\}|+|I\cap(\bigcup_{i=1}^{a}\{p_{i},p_{i}',q_{i}\})|+|I\cap \{x,y,z\})|\leq2a+b+2$,
\end{center}
a contradiction. Hence, $c\notin I$ and we have $|I|\leq3a+b+1$ due to the statements $(i)$, $(ii)$ and $(iii)$. This, together with $|I|\geq3a+b+1$, implies the desired equality. $(\square)$

\medskip
Now, assume that $f:U\longrightarrow \{True,False\}$ is a satisfying truth assignment for $C$. We set
\begin{center}
$M=\big{\{}c_{i}'\big{\}}_{i=1}^{b}\bigcup \big{\{}r_{i},s_{i}\big{\}}_{i=1}^{a}\bigcup \big{\{}z\big{\}}\bigcup \big{\{}p_{i}\mid i\in[a]\ \mbox{and}\ f(p_{i})=False\big{\}}\bigcup \big{\{}p_{i}'\mid i\in[a]\ \mbox{and}\ f(p_{i})=True\big{\}}$,
\end{center} 
and show that $M$ is an IMV set in $G$. It is clear from the structure that $M$ is independent in $G$. Consider any vertex $c_{j}'$ from $\{c_{1}',\ldots,c_{b}'\}$. For every $i\in[b]\setminus \{j\}$, $c_{j}'cc_{i}'$ is a geodesic between $c_{j}'$ and $c_{i}'$ passing through $c\notin M$. Also, $c_{j}'cr_{i}$ and $c_{j}'cs_{i}$ are a $c_{j}',r_{i}$-geodesic and a $c_{j}',s_{i}$-geodesic, respectively, for all $i\in[a]$, whose internal vertices do not belong to $M$.

Consider an arbitrary index $i\in[a]$. Suppose that $c_{j}$ has a neighbor $w\in \{p_{i},p_{i}'\}$. If $w\in M$, then $c_{j}'c_{j}w$ is a $c_{j}',w$-geodesic, in which $c_{j}\notin M$. If $w\notin M$, then $c_{j}'c_{j}ww'$ is a geodesic between $c_{j}'$ and $w'\in \{p_{i},p_{i}'\}\setminus \{w\}$ whose internal vertices are not contained in $M$. Assume now that $c_{j}$ does not have any neighbor in $\{p_{i},p_{i}'\}$ and that $w\in \{p_{i},p_{i}'\}\cap M$. It is clear from the structure that $d_{G}(c_{j}',w)=4$. Since $f$ is a satisfying truth assignment for $C$, it follows that $c_{j}$ is adjacent to a vertex $v\in \{p_{k},p_{k}'\}\cap(V(G)\setminus M)$ for some $k\in[a]$. In view of this, $c_{j}'c_{j}vxw$ turns out to be a $c_{j}',w$-geodesic with internal vertices not in $M$. Moreover, $c_{j}'c_{j}vxz$ is a $c_{j}',z$-geodesic of length $d_{G}(c_{j}',z)=4$ whose internal vertices do not belong to $M$. 

Let $i,j\in[a]$ be arbitrary distinct indices. It is easily checked that between any two vertices in $V(T_{i})\cap M$ there exists a geodesic, of length $2$ or $3$, whose internal vertices are not contained in $M$. Note that there is a geodesic of length $2$ with one endvertex in $\{r_{i},s_{i}\}$ (resp. $\{p_{i},p_{i}'\}\cap M$) and the other in $\{r_{j},s_{j}\}$ (resp. $\{p_{j},p_{j}'\}\cap M$) which passes through $y\notin M$ (resp. $x\notin M$). Moreover, there exists a geodesic of length $3$ passing through $x,y\notin M$ between any vertex in  $\{r_{i},s_{i}\}$ and the unique vertex in $\{p_{j},p_{j}'\}\cap M$.

Finally, for each vertex $w$ in $\{r_{i},s_{i}\}$ (resp. $\{p_{i},p_{i}'\}\cap M$), $wyz$ (resp. $wxz$) is a $w,z$-geodesic whose internal vertex is not contained in $M$. All in all, we have proved that $M$ is an IMV set in $G$ of cardinality $3a+b+1$. Therefore, $|M|=3a+b+1=\alpha(G)$ by Claim A. This immediately implies that $\mu_{i}(G)=\alpha(G)$ as $\alpha(G)\geq \mu_{i}(G)\geq|M|=3a+b+1$.

Conversely, assume that $\mu_{i}(G)=\alpha(G)$. In particular, $G$ has a $\mu_{i}(G)$-set $Q$ of cardinality $3a+b+1$ due to Claim A. Because $Q$ is an independent set in $G$, it follows that $Q$ fulfills the statements $(i)$, $(ii)$ and $(iii)$ mentioned in the proof of Claim A. Similarly to the proof of Claim A, we deduce that $Q$ has precisely the mentioned number of vertices in $(i)$, $(ii)$ and $(iii)$, respectively. The resulting equality $|Q\cap V(T_{i})|=3$, for each $i\in[a]$, necessarily implies that $r_{i},s_{i}\in Q$ and that precisely one of $p_{i}$ and $p_{i}'$ belongs to $Q$. In view of this, we deduce that $Q\cap \{c,x,y\}=\emptyset$, that $z\in Q$, and that exactly one vertex from $\{c_{j},c_{j}'\}$ belongs to $Q$ for each $j\in[b]$. On the other hand, we observe that $d_{G}(c_{j},z)=3$ and that all $c_{j},z$-geodesics pass through $\bigcup_{i=1}^{a} \{p_{i},p_{i}'\}$. Moreover, $d_{G}(c_{j}',z)=4$ and each $c_{j}',z$-geodesic passes through $\bigcup_{i=1}^{a} \{p_{i},p_{i}'\}$ or $\bigcup_{i=1}^{a} \{r_{i},s_{i}\}$. However, if $c_{j}'\in Q$, then $c_{j}'$ and $z$ are not mutually visible through a geodesic that intersects $\bigcup_{i=1}^{a} \{r_{i},s_{i}\}$.

Let $j\in[b]$ be arbitrarily chosen. In view of the above argument, and taking into account the facts that $Q$ is an MV set in $G$ with $|Q\cap \{c_{j},c_{j}'\}|=1$, we infer that $c_{j}$ is adjacent to at least one vertex in $(V(G)\setminus Q)\cap(\bigcup_{i=1}^{a}\{p_{i},p_{i}'\})$. With this in mind, $g:U\longrightarrow \{True,False\}$ defined by $g(u_{i})=True$ if and only if $p_{i}\notin Q$, turns out to be a satisfying truth assignment for $C$. This completes the proof.
\end{proof}


\section{Bounds and exact values}\label{BE}

Since the problem of finding the IMV chromatic number of graphs is NP-hard, as already shown, it is desirable to find tight lower and upper bounds for it or compute its exact value for specific graph classes. We hence center our attention in this section into bounding this parameter in terms of other graph parameters/invariants. 

\begin{theorem}\label{th:bound-tringle-free-mu_i}
If $G$ is a connected triangle-free graph of order $n\geq2$, then 
\begin{center}
$\chi_{\mu_{i}}(G)\le \left\lceil\dfrac{n-\mu_{i}(G)}{2}\right\rceil+1$,
\end{center}
and this bound is sharp.
\end{theorem}

\begin{proof}
Let $M$ be a $\mu_{i}(G)$-set of the triangle-free graph $G$. We set $G'=G-M$, which is a triangle-free graph as well. We need to differentiate two cases depending on the behavior of $G'$.\vspace{1mm}\\
\textit{Case 1}. $G'\cong K_{2}$. Let $V(G')=\{x,y\}$. If one of $x$ and $y$, say $x$, is of degree $1$ in $G$, then $M\cup \{x\}$ is an IMV set in $G$ of cardinality $|M|+1$, a contradiction. This shows that both $x$ and $y$ have neighbors in $M$. Since $G$ is triangle-free and because $xy\in E(G)$, it follows that $N_{G}(x)\cap N_{G}(y)=\emptyset$. Moreover, $N_{G}(x)\cup N_{G}(y)=V(G)$ as the graph $G$ is connected. In such a situation, we observe that $\{N_{G}(x),N_{G}(y)\}$ is an IMV coloring of $G$ of cardinality $2= \left\lceil\big{(}n-\mu_{i}(G)\big{)}/2\right\rceil+1$.\vspace{1mm}\\
\textit{Case 2}. $G'\ncong K_{2}$. If $n(G')=1$, then $G$ is necessarily isomorphic to a star. Therefore, $\chi_{\mu_{i}}(G)=2=\left\lceil\big{(}n-\mu_{i}(G)\big{)}/2\right\rceil+1$. So, we may assume that $n(G')\geq 2$. Since $G'\ncong K_{2}$ and because $G'$ is a triangle-free graph, there exist two non-adjacent vertices $u,v\in V(G')$ such that $G'-\{u,v\}\ncong K_{2}$ except when $G'\cong K_{1,3}$ and in this case we have three such vertices $u,v,w$ and $G'-\{u,v,w\}\cong K_1$. By iterating this process until the remaining graph is empty, we can partition $V(G')$ into sets $M_{1},\ldots,M_{k}$, for some positive integer $k$, such that\vspace{1mm}\\
$\bullet$ $|M_{k}|\in \{1,2\}$, and $|M_{i}|=2$ for each $i\in[k-2]$ (if any), \vspace{1mm}\\
$\bullet$ $|M_{k-1}|\in\{2,3\}$ (if any) where $|M_{k-1}|=3$ only when $G'\cong K_{1,3}$, and\vspace{1mm}\\ 
$\bullet$ $M_{i}$ is an independent set in $G$ for each $i\in[k]$.\vspace{1mm}

Note that each set $M_{i}$ is clearly an IMV set in the graph $G$. Therefore, $\mathcal{Q}=\{M,M_{1},\ldots,M_{k}\}$ turns out to be an IMV coloring of $G$. Thus, $\chi_{\mu_{i}}(G)\leq|\mathcal{Q}|=1+k$, which is less than or equal to the desired upper bound.

That the bound is sharp, may be seen as follows. Consider the graph $H=P\odot \overline{K_{r}}$, in which $r\geq2$ and $P=u_{1}\ldots u_{2k-1}$ is a path for some positive integer $k$. Let $H'$ be obtained from $H$ by joining a new vertex $u$ to all vertices of the $k$th copy of $\overline{K_{r}}$. The resulting graph $H'$ is clearly triangle-free. Note that $n(H')=(r+1)(2k-1)+1$ and that the union of vertex sets of the $2k-1$ copies of $\overline{K_{r}}$ is a $\mu_{i}(H')$-set of cardinality $r(2k-1)$. We define $c(u_{4i-3})=c(u_{4i-1})=2i-1$ and $c(u_{4i-2})=c(u_{4i})=2i$ for each $i$ such that the corresponding vertex is on $P$; $c(u)=c(u_{2k-1})$ if $k$ is odd; $c(u)=c(u_{2k-2})$ if $k$ is even; and $c(v)=0$ for any other vertices. It is readily observed that $c$ is a $\chi_{\mu_{i}}(H)$-coloring with $k+1=\lceil\big{(}n(H')-\mu_{i}(H')\big{)}/2\rceil+1$ colors. This completes the proof. 
\end{proof} 

In a triangle-free graph $G$, the open neighborhood of any vertex is an IMV set. Hence, it holds that $\mu_i(G)\ge \Delta(G)$, which leads to the following conclusion by using Theorem \ref{th:bound-tringle-free-mu_i}.

\begin{corollary}
\label{cor:bound-tringle-free-delta}
If $G$ is a connected triangle-free graph of order $n$, then $$\imv(G)\le \left\lceil\frac{n-\Delta(G)}{2}\right\rceil+1.$$
\end{corollary}

\begin{proposition}\label{Pro1}
If $G$ is a graph, then $$\imv(G)\le \chi(G)\mvc(G),$$       
and the bound is sharp.
\end{proposition}
\begin{proof}
Let $c:V(G)\rightarrow [\mvc(G)]$ be an MV coloring of $G$. Set $k=\chi(G)$. For each color class $C_i$, with $i\in [\mvc(G)]$, corresponding to $c$, consider the subgraph $H_i$ induced by $C_i$. With the resulting inequality $\chi(H_i)\le \chi(G)$ in mind, let $c_i:V(H_i)\rightarrow [k]$ be a proper coloring of the vertices in $H_i$. Now, consider the coloring $c':V(G)\rightarrow [\mvc(G)]\times [k]$ defined by $c'(v)=\big{(}c(v),c_{c(v)}(v)\big{)}$. To see that $c'$ is an IMV coloring of $G$, note that if two vertices $u$ and $v$ satisfy that $c(u)=c(v)=i$ for some $i\in [\mvc(G)]$, then $u$ and $v$ belong to the subgraph $H_i$. If $uv\notin E(G)$, then there is a $u,v$-geodesic in $G$ that avoids all vertices $w$ whose first component of $c'(w)$ equals $c(u)$. However, if $uv\in E(G)$, then their second components, $c_{i}(u)$ and $c_{i}(v)$, are different because $c_i$ is a proper coloring. Since $c'$ is an IMV coloring of $G$ with $\chi(G)\mvc(G)$ colors, the upper bound follows.  The bound is sharp for complete graphs, since $\imv(K_n)=n$, $\chi(K_n)=n$ and $\mvc(K_n)=1$. 
\end{proof}

As a consequence of Proposition \ref{Pro1}, we have the following upper bound for bipartite graphs.

\begin{corollary}
If $G$ is a bipartite graph, then $\imv(G)\le 2\mvc(G)$.    
\end{corollary}

In~\cite{KKVY}, two classes of graphs $G$ satisfying the property $\mvc(G)=2$ were found, but a complete characterization of such graphs seems to be a challenging problem. In the case of IMV coloring, we give the following characterization. 

\begin{theorem}
\label{thm:imv=2}    
If $G$ is a connected graph, then $\imv(G)=2$ if and only if $G$ is bipartite and $1\le {\rm diam}(G)\le 3$. 
\end{theorem}
\begin{proof}
Let $G$ be a connected graph with $\imv(G)=2$, and let $c:V(G)\rightarrow [2]$ be a $\imv(G)$-coloring. In particular, $G$ has at least one edge. So, $\chi(G)\ge 2$. Furthermore, $\chi(G)\le\imv(G)=2$ implies that $G$ is bipartite. Suppose to the contrary that ${\rm diam}(G)\ge 4$. Let $u$ and $v$ be two vertices in $G$ with $d_G(u,v)=4$. Let $uu_1u_2u_3v$ be any $u,v$-geodesic. Since $c$ is a proper $2$-coloring of $G$, we infer that $c(u)=c(u_2)=c(v)$ and that $c(u_1)=c(u_3)$. Since $\{g\in V(G)\mid c(g)=c(u)\}$ is an MV set, $\{g\in V(G)\mid c(g)=c(u_{1})\}$ is independent and $\chi_{\mu_{i}}(G)=2$, it follows that there exists a path $uwv$ in $G$ such that $c(w)=c(u_{1})$. This contradicts the fact that $d_{G}(u,v)=4$. Hence, $G$ is indeed a bipartite graph with ${\rm diam}(G)\in \{1,2,3\}$.

Conversely, $\chi(G)=2$ as $G$ is bipartite and $\imv(G)=\chi(G)=2$ by the remark before Proposition \ref{Pro}.
\end{proof}



Next, we turn our attention to another lower bound on the IMV chromatic number, namely $\imv(G)\ge \chi_\mu(G)$. We can prove it is also achieved by several well-known graph classes. In particular, this includes (almost all) cycles and trees.

\begin{proposition}\label{cycle}
If $n\ge 6$ or $n=4$, then $$\imv(C_n)=\chi_\mu(C_n)=\Big\lceil \frac{n}{3}\Big\rceil.$$
In addition, $\imv(C_n)=3$ if $n\in\{3,5\}$, $\chi_\mu(C_5)=2$ and $\chi_\mu(C_3)=1$.
\end{proposition}
\begin{proof}
Since $\mu(C_{n})=3$ (\cite{DiS}) and $\chi_\mu(G)\geq \lceil n(G)/\mu(G)\rceil$ for each graph $G$ (\cite{KKVY}), it follows by definitions that $\imv(C_n)\geq \chi_\mu(C_n)\geq \lceil n/3\rceil$. Assume that $n\geq6$ and $C_{n}=v_{1}v_{2}\ldots v_{n}v_{1}$.\vspace{0.5mm}\\
$\bullet$ If $n=3k$ for some positive integer $k$, then the sets $V_{i}=\{v_{i},v_{i+k},v_{i+2k}\}$, for $i\in[k]$, form an IMV coloring of $C_{n}$.\vspace{0.5mm}\\
$\bullet$ Assume that $n=3k+1$ for some positive integer $k$. Then, the sets $V_{i}=\{v_{i},v_{i+k},v_{i+2k}\}$, for $i\in[k]$, and $V_{k+1}=\{v_{3k+1}\}$ form an IMV coloring of $C_{n}$.\vspace{0.5mm}\\
$\bullet$ Finally, assume that $n=3k+2$ for some positive integer $k$. In such a situation, we observe that the sets $V_{i}=\{v_{i},v_{i+k},v_{i+2k}\}$ for $i\in[k-1]$, $V_{k}=\{v_{k},v_{2k},v_{3k+1}\}$ and $V_{k+1}=\{v_{3k},v_{3k+2}\}$ form an IMV coloring of $C_{n}$.\vspace{0.5mm}\\
In each case, the resulting IMV coloring has cardinality $\lceil n/3\rceil$. This leads to $\imv(C_n)=\chi_\mu(C_n)=\lceil n/3\rceil$ when $n\geq6$. On the other hand, it is readily verified that $\imv(C_4)=\chi_\mu(C_4)=\chi_\mu(C_5)=2$, $\imv(C_3)=\imv(C_5)=3$ and $\chi_\mu(C_3)=1$.
\end{proof}

Next, we consider trees. 

\begin{theorem}\label{thm:Cartesiantreescomplete}
If $T$ is a tree on at least three vertices, then $\imv(T)=\chi_\mu(T)$.    
\end{theorem}
\begin{proof}
Let $T$ be a tree of order $n\ge 3$, and let $c:V(T)\rightarrow [\chi_\mu(T)]$ be a mutual visibility coloring of $T$ such that $|\{uv\in E(T)\mid c(u)=c(v)\}|$ is minimized. Suppose that there exist two adjacent vertices $a$ and $b$ in $T$ with $c(a)=c(b)$. We may assume without loss of generality that $c(a)=c(b)=1$. Consider the sets $W_{ab}=\{u\in V(T)\mid d_T(u,a)<d_T(u,b)\}$ and $W_{ba}=\{u\in V(T)\mid d_T(u,b)<d_T(u,a)\}$. 
 
Since every path between any vertex in $W_{ab}$ (resp. $W_{ba}$) and $b$ (resp. $a$) passes through $a$ (resp. $b$) and because $c(a)=c(b)=1$, it follows that there is no vertex $x$ in $V(T)\setminus\{a,b\}$ with $c(x)=1$. Thus, $a$ and $b$ are the only vertices in $T$ to which $c$ assigns color $1$.  
 
Since $n\ge 3$, at least one of the sets $W_{ab}$ or $W_{ba}$, say $W_{ab}$, is not a singleton set. We may assume without loss of generality that there exist vertices in $W_{ab}$ to which $c$ assigns color $2$. We distinguish two cases.\vspace{1mm}\\
{\it Case 1.} For every $y\in W_{ab}$ with $c(y)=2$, every vertex $z$, different from $y$, lying on the $y,a$-geodesic has $c(z)\ne 2$.\\ 
In this case, the assignment $c'(a)=2$, $c'(v)=1$ for each $v\in W_{ab}\cap \big{(}c^{-1}(2)\big{)}$ and $c'(v)=c(v)$ for any other vertex $v$ defines an MV coloring of $T$ with $|\{uv\in E(T)\mid c'(u)=c'(v)\}|<|\{uv\in E(T)\mid c(u)=c(v)\}|$, a contradiction to our choice of $c$.\vspace{1mm}\\ 
{\it Case 2.} There exist distinct vertices $y,z\in W_{ab}$ with $c(y)=c(z)=2$ such that $z$ lies on the $y,a$-geodesic.\\
Let $u$ be the neighbor of $z$ on the $z,y$-geodesic. If there exists $x\in c^{-1}(2)\cap W_{zu}$ different from $z$, then the unique $x,y$-geodesic contains $z$, and so $x$ and $y$ are not mutually-visible vertices in $c^{-1}(2)$, a contradiction. Hence $c^{-1}(2)\cap W_{zu}=\emptyset$, and the assignment $c'(a)=2$, $c'(z)=1$ and $c'(v)=c(v)$ for any other vertex $v$ defines an MV coloring of $T$ with $|\{uv\in E(T)\mid c'(u)=c'(v)\}|<|\{uv\in E(T)\mid c(u)=c(v)\}|$, again a contradiction to the choice of $c$.

Since in either case we obtained a mutual visibility coloring $c'$ of $T$ for which the number of edges $uv\in E(T)$ with $c'(u)=c'(v)$ is smaller than the number of edges with $c(u)=c(v)$, we infer a contradiction with the choice of $c$. Therefore, $c$ has no such edges, and so $c$ is an IMV coloring. This yields $\imv(T)=\chi_{\mu}(T)$.
\end{proof}


\subsection{Diameter $2$ graphs}
\label{sec:diam2}

There are many graphs that attain the lower bound $\imv(G)\ge \chi(G)$. Due to the remark before Proposition \ref{Pro}, we have the following result.

\begin{remark}\label{rem:diam-3}
If $G$ is a graph with diameter at most $3$, then $\imv(G)=\chi(G)$.    
\end{remark}

Next, we turn our attention to the MV coloring. The following upper bound was found in~\cite{KKVY}.

\begin{proposition}\emph{(\cite[Proposition 2.2]{KKVY})}\label{prp:boundchromatic}
If $G$ is a graph of diameter $2$, then $\chi_\mu(G)\le \chi(G)$.
\end{proposition}

It is easy to see that for the complete bipartite graphs the equality is achieved in the bound of Proposition~\ref{prp:boundchromatic}; notably $\chi_\mu(K_{r,s})=2= \chi(K_{r,s})$. It was asked in~\cite[Problem 8.3]{KKVY} whether there are any other graphs that achieve this bound with the ultimate goal to characterize the diameter 2 graphs satisfying $\chi_\mu(G)=\chi(G)$. We present some partial results towards this goal.

We improve the bound from Proposition~\ref{prp:boundchromatic} by considering the concept of defective coloring defined as follows. Given a graph $G$ and integers $k\ge 1$ and $d\ge 0$, a $(k,d)$-\textit{coloring} of $G$ is a $k$-coloring of the vertices of $G$ such that each vertex $x$ has at most $d$ neighbors with the same color as $x$.  In other words, each color class in a $(k,d)$-coloring of $G$ induces a subgraph with maximum degree $d$. In particular, a $(k,0)$-coloring is exactly a proper $k$-coloring. The concept was introduced in the 1980s~\cite{AJ,CCW}, and has been considered from various aspects in a number of papers (see two recent studies~\cite{BLM,Liu}). 

We are interested in {\em $1$-defective $k$-colorings}, which are $(k,1)$-colorings. The smallest $k$ such that $G$ admits a $(k,1)$-coloring is the {\em $1$-defective chromatic number of $G$}, denoted $\chi_1(G)$. Clearly, $\chi_1(G)\le \chi(G)$ for any graph $G$, so the following result is an improvement of Proposition~\ref{prp:boundchromatic}.

\begin{proposition}\label{prp:bound1defective}
If $G$ is a graph of diameter $2$, then $\chi_\mu(G)\le \chi_1(G)$.
\end{proposition}
\begin{proof}
Let $G$ be a graph of diameter $2$, and let $c:V(G)\rightarrow [k]$ be a $(k,1)$-coloring of $G$ such that $k=\chi_1(G)$. Set $V_i=c^{-1}(i)$, for $i\in [k]$, as the color classes that arise from $c$, and consider arbitrary two vertices $x,y$ in $V_i$. If $xy\in E(G)$, then $x$ and $y$ are clearly $V_i$-visible since the unique $x,y$-geodesic does not have any internal vertices. On the other hand, if $xy\notin E(G)$, then since ${\rm diam}(G)=2$, there is a common neighbor $z$ of $x$ and $y$. Since $c$ is a $1$-defective coloring of $G$, $z$ belongs to $V_j$ for some $j\ne i$. Therefore, the $x,y$-geodesic $xzy$ has its only internal vertex outside of $V_i$. Therefore, $V_i$ is an MV set for every $i\in [k]$ and $c$ is an MV $k$-coloring. 
\end{proof}

In view of~\cite[Problem 8.3]{KKVY} searching for graphs with $\chi_\mu(G)=\chi(G)$, it is interesting to consider the graphs $G$ for which $\chi_\mu(G)=\chi_1(G)$ holds. (Since otherwise $\chi_\mu(G)<\chi_1(G)\le \chi(G)$.)

A graph $G$ is {\em geodetic} if every two vertices in $G$ are connected by a unique geodesic. Note that if ${\rm diam}(G)=2$, then $G$ is geodetic if and only if $G$ has neither $C_4$ nor $K_4-e$ as an induced subgraph. If $G$ is in turn triangle-free, then the bound in Proposition~\ref{prp:bound1defective} is attained. 

\begin{proposition}\label{prp:boundK3C4free}
If $G$ is a $\{C_3,C_4\}$-free graph of diameter $2$, then $\chi_\mu(G)=\chi_1(G)$.
\end{proposition}
\begin{proof}
Let $G$ be a $\{C_3,C_4\}$-free graph of diameter $2$. Let $c:V(G)\rightarrow [k]$, where $k=\chi_\mu(G)$ be an MV coloring of $G$. Suppose that $c$ is not a $1$-defective coloring of $G$. Then, there exists a color class $V_i=c^{-1}(i)$ such that $G[V_i]$ has a vertex $x$ with degree at least $2$. Let $\{u,v\}\subseteq N_G(x)\cap V_i$. Since $G$ is triangle-free, the path $P=uxv$ is a $u,v$-geodesic whose internal vertex $x$ belong to $V_i$. Since $G$ is $C_4$-free, there is no other $u,v$-geodesic apart from $P$, and we are in a contradiction to the fact that $c$ is an MV coloring of $G$. Thus, $c$ is a $1$-defective coloring of $G$, and so $\chi_1(G)\le \chi_\mu(G)$. By Proposition~\ref{prp:bound1defective}, we derive the desired equality.
\end{proof}

Triangle-free geodetic graphs with diameter $2$ form a small family, but have surprisingly appeared in many studies. Due to~\cite[Result II]{st-74}, they are precisely the Moore graphs with diameter 2 (and girth 5). There are only three known such graphs: $C_5$, the Petersen graph and the Hoffman-Singleton graph, which is a $7$-regular graph on $50$ vertices. The only other possible candidates are regular graphs with degree 57 on 3250 vertices, but it is a well-known open problem whether there are any such graphs. It is not hard to see that $\chi_\mu(C_5)=2<3=\chi(C_5)$, and the same holds for the Petersen graph $G$ with $\chi_\mu(G)=2<3=\chi(G)$. It is known that for the Hoffman-Singleton graph $H$, $\chi(H)=4$. Hence, $\chi_1(H)\le 4$, while its exact value would also yield the MV chromatic number of $H$ due to Proposition~\ref{prp:boundK3C4free}. 


\subsection{Lexicographic product graphs}

Several of our results (see Theorems \ref{th:subdiv-K_n}, \ref{th:subdiv-K_r-s}, \ref{thm:Cartesiantreescomplete} and Proposition \ref{cycle} as well as the forthcoming Theorems \ref{strong-paths} and \ref{th-mvc-strong-paths}) indicate that $\chi_{\mu}(G)$ and $\chi_{\mu_{i}}(G)$ are rather close to each other or even equal. In contrast, the results of this subsection show that they can be arbitrarily apart as well.

Clearly, $G\circ H\cong K_{n(G)n(H)}$ if and only if $G\cong K_{n(G)}$ and $H\cong K_{n(H)}$. In such a situation, $\chi_{\mu}(G\circ H)=1$. Moreover, $G\circ K_{1}\cong G$ and $K_{1}\circ H\cong H$. So, we may assume in the next theorem that at least one of $G$ and $H$ is not complete, and that both have at least two vertices.
We first recall, see \cite{HIK}, that the distance between any two vertices $(g,h),(g',h')\in V(G\circ H)$, where $G$ has no isolated vertices, is given by
\begin{equation}\label{formula}
d_{G\circ H}\big{(}(g,h),(g',h')\big{)}=\left \{
\begin{array}{lll}
\min\{2,d_{H}(h,h')\} & \mbox{if}\ g=g',\\
d_{G}(g,g') & \mbox{if}\ g\neq g'.
\end{array}
\right.
\end{equation}

\begin{theorem}\label{lex-mut-col}
Let $G$ be a connected graph and $H$ be a graph, both on at least two vertices. If at least one of them is not complete, then $\chi_{\mu}(G\circ H)=2$.
\end{theorem}
\begin{proof}
Let $h$ and $h^{*}$ be distinct vertices of $H$. Setting $\overline{G^{h^*}}=V(G\circ H)\setminus G^{h^*}$,  
we show that $\mathcal{Q}=\{G^{h^*},\overline{G^{h^*}}\}$ is an MV coloring of $G\circ H$. Let $(g,h^{*}),(g',h^{*})\in G^{h^{*}}$ be distinct. Let $gg_{1}\ldots g_{k}g'$ be a $g,g'$-geodesic in $G$. Using the distance formula (\ref{formula}) it follows that $(g,h^{*})(g_{1},h)\ldots(g_{k},h)(g',h^{*})$ is a $(g,h^{*}),(g',h^{*})$-geodesic in $G\circ H$ such that $\{(g_{1},h),\ldots,(g_{k},h)\}\cap G^{h^{*}}=\emptyset$. Therefore, $G^{h^{*}}$ is an MV set in $G\circ H$.

Now, let $(g,h)$ and $(g',h')$ be distinct vertices in $\overline{G^{h^*}}$ and we may assume that they are not adjacent. Assume first that $g=g'$. If $gg^{*}\in E(G)$, then $(g,h)(g^{*},h^{*})(g,h')$ is a $(g,h),(g,h')$-geodesic in $G\circ H$ whose internal vertex is not contained in $\overline{G^{h^*}}$. So, we may assume that $g\neq g'$. By taking any $g,g'$-geodesic $gg_{1}\ldots g_{k}g'$ in $G$, we get the $(g,h),(g',h')$-geodesic $(g,h)(g_{1},h^{*})\ldots(g_{k},h^{*})(g',h')$ whose internal vertices do not belong to $\overline{G^{h^*}}$. Hence, $\overline{G^{h^*}}$ is an MV set in $G\circ H$ as well. In either case, we have proved that $\mathcal{Q}$ is an MV coloring of $G\circ H$. Thus, $\chi_{\mu}(G\circ H)\leq|\mathcal{Q}|=2$. This results in $\chi_{\mu}(G\circ H)=2$ as $G\circ H$ is not a complete graph.
\end{proof}

The following lemma will be useful to determine the IMV (chromatic) number of lexicographic product graphs.

\begin{lemma}\label{lex-indep}
If $G\ncong K_1$ is a connected graph and $H$ a graph with at least one edge, then every independent set in $G\circ H$ is an IMV set in $G\circ H$.
\end{lemma}
\begin{proof}
Let $I$ be an independent set in $G\circ H$. Let $(g,h),(g',h')\in I$ be distinct. Assume first that $g=g'$. Since $G\ncong K_1$ is connected, $gg^{*}\in E(G)$ for some $g^{*}\in V(G)$. On the other hand, because $I$ is an independent set in $G\circ H$, it follows that $(g^{*},h^{*})\notin I$ for each $h^{*}\in V(H)$. This argument results in the existence of the $(g,h),(g,h')$-geodesic $(g,h)(g^{*},h^{*})(g,h')$ for any $h^{*}\in V(H)$.

Assume now that $g\neq g'$ and let $xy\in E(H)$. Let $gg_{1}\ldots g_{k}g'$ be a $g,g'$-geodesic in $G$. Since $I$ is an independent set in $G\circ H$, at most one of $(g_i,x)$ and $(g_i,y)$ belongs to $I$ for every $i\in [k]$. Let $h^*_i$, with $i\in [k]$, be such that $(g_i,h_i^*)\notin I$. So, the internal vertices of the $(g,h),(g',h')$-geodesic $(g,h)(g_{1},h^{*}_{1})\ldots(g_{k},h^{*}_{k})(g',h')$ do not belong to $I$. Indeed, we have proved that every independent set in $G\circ H$ is an MV set as well. 
\end{proof}

\begin{theorem}\label{Equ-color}
If $G\ncong K_1$ is a connected graph and $H$ is a graph with at least one edge, then the following statements hold.\vspace{1mm}\\
$(i)$ $\chi_{\mu_{i}}(G\circ H)=\chi(G\circ H)$, and\vspace{1mm}\\
$(ii)$ $\mu_{i}(G\circ H)=\alpha(G)\alpha(H)$. 
\end{theorem}

\begin{proof}
$(i)$ The inequality $\chi_{\mu_{i}}(G\circ H)\geq \chi(G\circ H)$ follows by~\eqref{eq:first-ineq}. Conversely, every independent set in $G\circ H$ is an MV set by Lemma \ref{lex-indep}. We infer that every color class of a $\chi(G\circ H)$-coloring is also an MV set and $\chi_{\mu_{i}}(G\circ H)\leq \chi(G\circ H)$ holds as well.

$(ii)$ This is an immediate consequence of Lemma \ref{lex-indep} by taking into account the fact that $\alpha(G\circ H)=\alpha(G)\alpha(H)$. 
\end{proof}


\section{Cartesian and strong products}\label{Cart-Str}

\subsection{Cartesian product}

Let $G$ and $H$ be any (connected) graphs, and $G\square H$ be their Cartesian product. Note that any $G$-fiber (resp.\ $H$-fiber) in $G\square H$ is isomorphic to $G$ (resp.\ $H$) and is also a convex subgraph of $G\square H$. Thus, using Lemma~\ref{lem:convex}, we infer the following result.

\begin{proposition}\label{prop:Cartesianlower}
If $G$ and $H$ are connected graphs, then $\chi_{\mu}(G\square H)\ge \max\{\chi_{\mu}(G),\chi_{\mu}(H)\}$, and the bound is sharp.
\end{proposition}

To see the sharpness of the bound in the proposition above, one can take $G=K_{r,s}$ and $H=K_n$, where max$\{r,s\}\geq2$ and $n\geq1$. Consider the partition of $V(K_{r,s})\times V(K_n)$ as follows: ${\cal S}=\{A\times V(K_n),B\times V(K_n)\}$, where $A$ and $B$ are partite sets of $K_{r,s}$ with $|A|=r$ and $|B|=s$. Note that $\cal S$ yields a mutual visibility $2$-coloring of $G\square H$ showing that $\chi_\mu(K_{r,s}\square K_n)=2=\max\{\chi_{\mu}(K_{r,s}),\chi_{\mu}(K_n)\}$.

In the following result we establish a connection between the MV coloring of the generalized Cartesian prisms of a graph $G$ (that is, $G\Box K_n$) with the IMV coloring of $G$. 

\begin{theorem}\label{prop:Cartesianupper}
If $G$ is a connected graph and $n$ a positive integer, then $\chi_{\mu}(G\square K_n)\le \imv(G)$, and the bound is sharp.
\end{theorem}
\begin{proof}
Let ${\cal S}=\{S_1,\ldots, S_k\}$ be a partition of $G$ into color classes of an IMV $k$-coloring of $G$ with $k=\imv(G).$
It suffices to verify that the partition ${\cal S}'=\{S_1\times V(K_n),\ldots, S_k\times V(K_n)\}$ yields an MV coloring of $G\square K_n$. To do so, for any $i\in[k]$, let $(g,h),(g',h')\in S_{i}\times V(K_{n})$ be distinct vertices. If $g=g'$, then $(g,h)(g',h')\in E(G\square K_{n})$. We may thus assume that $g\neq g'$. Since $S_{i}$ is simultaneously an MV set and an independent set in $G$, it follows that there exists a $g,g'$-geodesic $gg_{1}\ldots g_{r}g'$ in $G$ such that $r\geq1$ and that $g_{j}\notin S_{i}$ for every $j\in[r]$. If $h=h'$, then $(g,h)(g_{1},h)\ldots(g_{r},h)(g',h')$ is a $(g,h),(g',h')$-geodesic in $G\square K_{n}$ having no internal vertices from $S_{i}\times V(K_{n})$. On the other hand, if $h\neq h'$, then $P=(g,h)(g_{1},h)\ldots(g_{r},h)(g_{r},h')(g',h')$ is a $(g,h),(g',h')$-geodesic in $G\square K_{n}$ due to the fact that $d_{G\square K_{n}}\big{(}(g,h),(g',h')\big{)}=d_{G}(g,g')+1=r+2$. Moreover, no internal vertices of $P$ belong to $S_{i}\times V(K_{n})$. All in all, we have verified that the partition ${\cal S}'$ yields an MV coloring of $G\square K_{n}$, as desired. 
\end{proof}

Combining Proposition~\ref{prop:Cartesianlower} and Theorems~\ref{prop:Cartesianupper} and~\ref{thm:Cartesiantreescomplete}, we infer the following exact value.

\begin{corollary}
If $T$ is a tree on at least $3$ vertices, then $\chi_{\mu}(T\Box K_n)=\chi_{\mu}(T)$ holds for any positive integer $n$.
\end{corollary} 

We end this subsection with the following statement. 
\begin{proposition}
\label{prp:Hamming}
If $r,s,t\ge 2$ are integers, then $\imv(K_r\square K_s\square K_t)=\max\{r,s,t\}$.    
\end{proposition}

\begin{proof}
Since $K_r\square K_s \square K_t$ has diameter at most $3$ and $\chi(K_r\square K_s\square K_t)=\max\{r,s,t\}$, the conclusion follows from Remark~\ref{rem:diam-3}.
\end{proof}

\subsection{Strong product}

Recall, see \cite{HIK}, that the distance between two vertices $(g,h),(g',h')\in V(G\boxtimes H)$ is given by
\begin{equation}\label{Dis-Str}
d_{G\boxtimes H}\big{(}(g,h),(g',h')\big{)}=\max \{d_{G}(g,g'),d_{H}(h,h')\}.
\end{equation}

We illustrate an application of Lemma \ref{lem:convex} also in the case of strong product of graphs. Note that convex subgraphs are completely described in strong product graphs (\cite{Pete}). However, for our purpose, we restrict our attention to two special cases of convex subgraphs in $G\boxtimes H$. 

Let $g_1\dots g_k$ and $h_1\dots h_k$ be convex paths of $G$ and $H$, respectively. The subgraph of $G\boxtimes H$ induced by $\{(g_1,h_1),\dots,(g_k,h_k)\}$ is a \emph{diagonal}. It is not hard to see (and it follows from~\cite[Theorem 3.8]{Pete}) that diagonals are convex. In order to obtain a diagonal of maximum length, we need to take a convex path of maximum length in each factor. By ${\rm cp}(G)$ we denote the length of a longest convex path in $G$ and ${\rm diag}(G\boxtimes H)$ is the length of a longest diagonal in $G\boxtimes H$. In fact, ${\rm diag}(G\boxtimes H)=\min\{{\rm cp}(G),{\rm cp}(H)\}$. These concepts enable us to prove a lower bound for $\imv(G\boxtimes H)$. Recall that $\omega(G)$ denotes the cardinality of a maximum clique in $G$ and that an edge $e$ in $E(G\square H)$ \big{(}resp. $E(G\boxtimes H)\setminus E(G\square H)$\big{)} is called a~\textit{Cartesian edge} (resp.~\textit{non-Cartesian edge}) in $G\boxtimes H$. 

\begin{theorem}\label{strong} 
If $G$ and $H$ are connected graphs, then 
\begin{center}
$\imv(G\boxtimes H)\geq \max\left\{\left\lceil \dfrac{{\rm diag}(G\boxtimes H)}{2}\right\rceil\!,\,\omega(G)\omega(H)\right\}$.
\end{center}
\end{theorem}
\begin{proof}
Firstly, we have
$$\imv(G\boxtimes H)\geq\chi(G\boxtimes H)\geq \omega(G\boxtimes H)=\omega(G)\omega(H).$$

On the other hand, let $D$ be a diagonal of $G\boxtimes H$ with $|D|={\rm diag}(G\boxtimes H)$. Since $D\cong P_{|D|}$ is a convex subgraph of $G\boxtimes H$, it follows by Lemma \ref{lem:convex} that $\imv(G\boxtimes H)\geq \imv(P_{|D|})=\lceil|D|/2\rceil=\lceil{\rm diag}(G\boxtimes H)/2\rceil$.       
\end{proof}

To see that the bound of Theorem~\ref{strong} is sharp and that both lower bounds can be attained, observe 
the following result on paths. 

\begin{theorem}\label{strong-paths}
If $t\geq 4k\geq 8$, then $\imv(P_t\boxtimes P_{4k})=\frac{{\rm diag}(P_t\boxtimes P_{4k})}{2}=2k$ and $\imv(P_t\boxtimes P_r)=4$ for $t\geq r\in \{3,4,5,6,7\}$.
\end{theorem}

\begin{proof}
By Theorem \ref{strong}, we have $\imv(P_t\boxtimes P_{4k})\geq {\rm diag}(P_{t}\boxtimes P_{4k})/2=2k$. For the reverse inequality, we first define $c:V(P_{4k}\boxtimes P_{4k})\rightarrow [2k]$ as follows. Consider $V(P_{4k}\boxtimes P_{4k})=\{(i,j)\mid i,j\in[4k]\}$ in a matrix form and let $C_j=\{(i,j)\mid i\in[4k]\}$ be the set of vertices of the $j$th column. For each $\ell\in [k]$, there are exactly four columns $C_{j}$ with $j\equiv \ell$ (mod $k$), namely $C_{\ell},C_{k+\ell},C_{2k+\ell},C_{3k+\ell}$. Now, for each of these columns, we color their vertices by using the colors $2\ell-1$ and $2\ell$ as follows. Consider the diagonal $D=\{(j,j)\mid j\in [4k]\}$. For every $\ell\in[k]$, the vertices in $D\cap C_{\ell}$ and $D\cap C_{3k+\ell}$ get the color $2\ell-1$ and the vertices in $D\cap C_{k+\ell}$ and $D\cap C_{2k+\ell}$ the color $2\ell$.  The rest of the coloring in columns $C_{pk+\ell}$, with $p\in\{0,1,2,3\}$, is now fixed due to the facts that $c$ is a proper coloring and that we only use the colors $2\ell-1$ and $2\ell$ in each of these columns. In particular, notice that any color is present in exactly four columns. See the described colorings for $P_{12}\boxtimes P_{12}$ and $P_8\boxtimes P_8$ in the following schemes.     
\begin{equation*}
\begin{array}{cccccccccccccccccccc}
1 & 4 & 5 & 1 & 4 & 5 & 2 & 3 & 6 & 2 & 3 & 6	& & & & & & & &	\\ 
2 & 3 & 6 & 2 & 3 & 6 & 1 & 4 & 5 & 1 & 4 & 5	& & & & & & & &	\\ 
1 & 4 & 5 & 1 & 4 & 5 & 2 & 3 & 6 & 2 & 3 & 6   &\ \ \ \    1 & 4 & 2 & 3 & 2 & 3 & 1 & 4\\ 
2 & 3 & 6 & 2 & 3 & 6 & 1 & 4 & 5 & 1 & 4 & 5	&\ \ \ \	2 & 3 & 1 & 4 & 1 & 4 & 2 & 3\\ 
1 & 4 & 5 & 1 & 4 & 5 & 2 & 3 & 6 & 2 & 3 & 6	&\ \ \ \	1 & 4 & 2 & 3 & 2 & 3 & 1 & 4\\ 
2 & 3 & 6 & 2 & 3 & 6 & 1 & 4 & 5 & 1 & 4 & 5	&\ \ \ \	2 & 3 & 1 & 4 & 1 & 4 & 2 & 3\\
1 & 4 & 5 & 1 & 4 & 5 & 2 & 3 & 6 & 2 & 3 & 6	&\ \ \ \	1 & 4 & 2 & 3 & 2 & 3 & 1 & 4\\ 
2 & 3 & 6 & 2 & 3 & 6 & 1 & 4 & 5 & 1 & 4 & 5	&\ \ \ \	2 & 3 & 1 & 4 & 1 & 4 & 2 & 3\\ 
1 & 4 & 5 & 1 & 4 & 5 & 2 & 3 & 6 & 2 & 3 & 6	&\ \ \ \	1 & 4 & 2 & 3 & 2 & 3 & 1 & 4\\ 
2 & 3 & 6 & 2 & 3 & 6 & 1 & 4 & 5 & 1 & 4 & 5	&\ \ \ \	2 & 3 & 1 & 4 & 1 & 4 & 2 & 3\\ 
1 & 4 & 5 & 1 & 4 & 5 & 2 & 3 & 6 & 2 & 3 & 6 & & & & & & & &\\ 
2 & 3 & 6 & 2 & 3 & 6 & 1 & 4 & 5 & 1 & 4 & 5 & & & & & & & &%
\end{array}%
\end{equation*}

For $P_t\boxtimes P_{4k}$ with $t>4k$, we continue with the described pattern in every column. The resulting coloring is a proper one, and we only need to show that it is an MV coloring as well. To do so, we will make use of some specific paths to verify the mutual-visibility property. Assume that $c(i,j)=c(p,r)$, where $i\leq p$ without loss of generality. There are exactly four columns which contain the color $c(i,j)$, and let them be ordered as $C(1)$, $C(2)$, $C(3)$ and $C(4)$. We will describe different paths that assure the mutual-visibility property according to which of the four columns $(i,j)$ and $(p,r)$ belong. For this, we will use a $(i,j),(p,r)$-geodesic $P$ (resp. $P'$) that starts at $(i,j)$ (resp. at $(p,r)$) with non-Cartesian edges and ends with the maximum number of Cartesian edges.

First, let $j=r$, which means that both vertices are in the same column, and $i<p$. If $j\neq 4k$, then the path $(i,j)(i+1,j+1)(i+2,j+1)\dots(p-1,j+1)(p,r)$ has no internal vertices of color $c(i,j)$. If $j=4k$, then a desired path is $(i,j)(i+1,j-1)(i+2,j-1)\dots(p-1,j-1)(p,r)$. 

Now, assume $(i,j)$ and $(p,r)$ belong to the columns $C(\ell)$ and $C(\ell+1)$, where $\ell\in[3]$. If $|j-r|\geq p-i$, then $P$ is the desired path. Otherwise, a path $(i,j)(i+1,j+1)\dots(i+r-j-1,r-1)(i+r-j,r-1)\dots(p-1,r-1)(p,r)$ when $j<r$ and a path $(i,j)(i+1,j-1)\dots(i+r-j+1,r+1)(i+r-j+2,r+1)\dots(p-1,r+1)(p,r)$ when $j>r$ are geodesics without internal vertices of color $c(i,j)$.

We continue with the case where $(i,j)$ and $(p,r)$ belong to columns $C(\ell)$ and $C(\ell+2)$ where $\ell\in[2]$. If $(i,j)$ belongs to $C(1)$ or $C(4)$ and $|j-r|\geq p-i$, then $P$ is without internal vertices of color $c(i,j)$. Similarly, if $(i,j)$ belongs to $C(2)$ or $C(3)$ and $|j-r|\geq p-i$, then $P'$ is without internal vertices of color $c(i,j)$. So, let $|j-r|<p-i$. If $(i,j)$ belongs to $C(1)$, then a desired path is $(i,j)\xrightarrow{P}(i+r-j-1,r-1)(i+r-j,r-1)\dots (p-1,r-1)(p,r)$ and if $(i,j)$ belongs to $C(4)$, then a desired path is $(i,j)\xrightarrow{P}(i+r-j+1,r+1)(i+r-j,r+1)\dots (p-1,r+1)(p,r)$. If $(i,j)$ belongs to $C(2)$, then a desired path is $(p,r)\xrightarrow{P'}(p+j-r+1,j+1)(i+j-r,j+1)\dots (i+1,j+1)(i,j)$ and if $(i,j)$ belongs to $C(3)$, then a desired path is $(p,r)\xrightarrow{P'}(p+j-r-1,j-1)(p+j-r,j-1)\dots (i+1,j-1)(i,j)$. 

Finally, let $(i,j)$ and $(p,r)$ belong to columns $C(1)$ and $C(4)$. If $|r-j|\leq p-i$, then we are done with a path $(i,j)\xrightarrow{P}(i+r-j-1,r-1)(i+r-j,r-1)\dots(p-1,r-1)(p,r)$ when $(i,j)\in C(1)$ and with a path $(i,j)\xrightarrow{P}(i+r-j+1,r+1)(i+r-j+2,r+1)\dots(p-1,r+1)(p,r)$ when $(i,j)\in C(4)$. So, let $|r-j|>p-i$. Now $P$ contains at most one vertex of color $c(i,j)$ and this vertex (if exists) is incident to two Cartesian edges of $P$. Say that $(s,t)$ is the internal vertex of $P$ with $c(s,t)=c(i,j)$. If we replace $(s,t)$ on $P$ with $(s-1,t)$, then we obtain a $(i,j)(p,r)$-geodesic without an internal vertex of color $c(i,j)$, which ends the first equality. 

To end the proof, let $t\geq r\in\{3,4,5,6,7\}$. By Theorem \ref{strong}, we have $\imv(P_t\boxtimes P_r)\geq 4$ as $\omega (P_t)=\omega (P_r)=2$. It is easy to obtain an IMV coloring of $P_r\boxtimes P_r$ by deleting the last $8-r$ rows and columns from the $(P_8\boxtimes P_8)$-coloring given in the previous scheme. Now, one can continue with the pattern in every column to obtain an IMV coloring of $P_t\boxtimes P_r$ with four colors. Hence, $\imv(P_t\boxtimes P_r)=4$. 
\end{proof}

In contrast to the above result, we next give the exact formula for $\mvc(P_t\boxtimes P_r)$. It is known that $\mvc(P_{n})=\left\lceil\frac{n}{2}\right\rceil$ for each $n\geq1$ (see \cite{KKVY}). With this in mind, we assume that min$\{t,r\}\geq2$.

\begin{theorem}\label{th-mvc-strong-paths}
Let $t$ and $r$ be two integers with min$\{t,r\}\geq2$. Then,
\begin{center}
$\mvc(P_t\boxtimes P_r)=\left\{\begin{array}{ll}
1; & \mbox{if $t=r=2$}, \\[0.15cm]
2; & \mbox{if $r\neq t$ and min$\{t,r\}=2$}, \\[0.15cm]
min\big{\{}\left\lceil\frac{t}{2}\right\rceil,\left\lceil\frac{r}{2}\right\rceil\big{\}}; & \mbox{otherwise}.
\end{array}\right.$
\end{center}
\end{theorem}
\begin{proof}
If $r=t=2$, then clearly $P_t\boxtimes P_r\cong K_4$, and the conclusion is trivial. From now on, by symmetry, we may assume that $t\geq r$. If $r=2$ and $t\ge 3$, then $P_t\boxtimes P_2\cong P_t\circ K_2$ and the result follows by Theorem \ref{lex-mut-col}. Due to this, we may assume that $r\geq3$.

By using similar arguments as the ones used in the proof of Theorem \ref{strong}, if $D$ is a diagonal of $P_t\boxtimes P_r$ with $|D|={\rm diag}(P_t\boxtimes P_r)$, then there can be at most two vertices of the same color on $D$, for any MV coloring of $P_t\boxtimes P_r$, because $(P_t\boxtimes P_r)[D]\cong P_{|D|}$ is a convex subgraph of $P_t\boxtimes P_r$. Thus, $\mvc(P_t\boxtimes P_r)\geq \left\lceil \frac{{\rm diag}(P_t\boxtimes P_r)}{2}\right\rceil=\left\lceil\frac{r}{2}\right\rceil$ follows from Lemma \ref{lem:convex}. 

Recall that by $C_{j}$, with $j\in[r]$, we mean the set of vertices in the $j$th column of the matrix form of $P_t\boxtimes P_r$. Assume first that $r=3$. Then, the assignment $1$ and $2$ to the vertices in $C_{1}\cup C_{3}$ and $C_{2}$, respectively, gives an MV coloring of $P_t\boxtimes P_3$. Now, let $r\geq4$.
If $r$ is even, we give the color $i$ to all vertices in $C_{i}\cup C_{i+\frac{r}{2}}$ for each $i\in\big{[}\frac{r}{2}\big{]}$. Due to the structure of $P_t\boxtimes P_r$ and (\ref{Dis-Str}), we infer that $c$ is an MV coloring of $P_t\boxtimes P_r$ with $\frac{r}{2}$ colors. Finally, if $r$ is odd, we first consider the subgraph $P_t\boxtimes P_{r-1}$. Note that the above process leads to an MV coloring $c$ of $P_t\boxtimes P_{r-1}$ with $\frac{r-1}{2}$ colors. It is now easy to see that $c'(v)=c(v)$ for each $v\in V(P_t\boxtimes P_{r-1})$, and $c'(v)=\frac{r+1}{2}$ for each $v\in C_{r}$ defines an MV coloring of $P_t\boxtimes P_r$ with $\frac{r+1}{2}$ colors. So, $\mvc(P_t\boxtimes P_r)\leq \left\lceil\frac{r}{2}\right\rceil$ in either case, which completes the proof.
\end{proof}

\section{Concluding remarks}

This paper is, in a sense, a continuation of the paper of Klav\v zar et al.~\cite{KKVY} in which mutual-visibility coloring was introduced and studied. In Section~\ref{sec:diam2}, we improved an upper bound from~\cite{KKVY} on $\mvc(G)$ for diameter $2$ graphs $G$ by involving $1$-defective chromatic number. The authors of the earlier paper suggested in~\cite[Problem 8.1]{KKVY} to investigate the computational complexity of computing the mutual-visibility chromatic number, which we also implicitly addressed. In Theorem~\ref{Complexity2} we proved that the decision problem arising from the independent version of mutual-visibility coloring is NP-complete (even in graphs with universal vertices). We strongly suspect that the same holds for the basic version of the mutual-visibility coloring. 

\begin{conjecture}
{\sc Mutual-Visibility Coloring} is NP-complete.
\end{conjecture}

Another argument that implicitly supports the conjecture arises from Theorem~\ref{th:subdiv-K_n} and Observation~\ref{obs:Ramsey}, in which we establish that $\mvc\big{(}S(K_n)\big{)}\in\{\rho(n),\rho(n)+1\}$, and $\rho(n)$ is in close relation with some Ramsey numbers, which are well known to be difficult. Concerning Theorem~\ref{th:subdiv-K_n}, we actually think that $\mvc\big{(}S(K_n)\big{)}=\rho(n)+1$ is likely to be true, while for the independent version of the coloring we even dare to conjecture it:

\begin{conjecture}
For any positive integer $n$, $\imv\big{(}S(K_n)\big{)}=\rho(n)+1$. 
\end{conjecture}

Results from Section \ref{sec:subdivision} suggest making a more complete study of the (I)MV chromatic numbers of the subdivision graphs of other graph classes. This might lead to finding some other relationships between the considered parameters and other classical topics of graph theory.

\begin{problem}
Study the (I)MV chromatic number of the subdivision graph of (an arbitrary) graph. 
\end{problem}

In relation with Proposition~\ref{Pro1}, it would be interesting to characterize the graphs $G$ 
that attain the equality $\imv(G)=\chi(G)\mvc(G)$. We wonder if there are any other connected graphs besides complete graphs that attain this bound. 

Trees on at least $3$ vertices and (most of the) cycles are families of graphs in which the two versions of mutual-visibility chromatic numbers have the same value. The following problem might be difficult, but obtaining partial solutions (such as finding other nice families with that property) seems achievable. 

\begin{problem}
Characterize the graphs $G$ with $\imv(G)=\mvc(G)$. 
\end{problem}

Proposition~\ref{prp:Hamming} establishes the IMV chromatic number of specific Hamming graphs with only three factors, in which case the number coincides with the chromatic number. It would be interesting to consider 
the family of Hamming graphs $(K_n)^{\Box,k}$, defined as $K_n\Box\cdots\Box K_n$, where there are $k$ factors of $K_n$, which is our last problem.

\begin{problem}
Determine $\imv((K_n)^{\Box,k})$ for all integers $n\ge 2$ and $k\ge 4$. In particular, for which $n$ and $k$ is $\imv((K_n)^{\Box,k})=\chi((K_n)^{\Box,k})$? 
\end{problem}

An interesting particular case in the problem above is when $n=2$ is considered. That is, studying the case of hypercube graphs. For instance, some studies for such classical family are already known from \cite{Axenovich}, for the case of the MV chromatic number.

\section*{Acknowledgments}
B.~Bre\v sar acknowledges the financial support from the Slovenian Research and Innovation Agency (research core funding No.\ P1-0297 and project grants N1-0285 and J1-4008). 
I. Peterin has been partially supported by Slovenian Research and Innovation Agency under the program P1-0297.
I. G. Yero has been partially supported by the Spanish Ministry of Science and Innovation through the grant PID2023-146643NB-I00. 


\end{document}